\documentclass{amsart}
\usepackage{amsmath, amssymb}
\usepackage[colorlinks=true]{hyperref}

\theoremstyle{plain}
 \newtheorem{theorem}{Theorem}[section]
 \newtheorem{lemma}{Lemma}[section]
 \newtheorem{proposition}{Proposition}[section]

\theoremstyle{definition}
 \newtheorem{remark}[theorem]{Remark}

\DeclareMathOperator{\Impart}{Im}
\DeclareMathOperator{\Gal}{Gal}
\newcommand{\Z}{\mathbb{Z}}
\newcommand{\Q}{\mathbb{Q}}
\newcommand{\R}{\mathbb{R}}
\newcommand{\C}{\mathbb{C}}

\title{An explicit Andr\'e--Oort type result for $\mathbb{P}^1(\C) \times \mathbb{G}_m(\C)$}

\author{Roland Paulin}
\address{Roland Paulin, Department of Mathematics, University of Salzburg, Hellbrunnerstr.\ 34/I, 5020 Salzburg, Austria}
\email{paulinroland@gmail.com}
\thanks{The author was supported by the Austrian Science Fund (FWF): P24574.}

\date{\today}

\begin{document}

\begin{abstract}
Using class field theory we prove an explicit result of Andr\'e--Oort type for $\mathbb{P}^1(\C) \times \mathbb{G}_m(\C)$.
In this variation the special points of $\mathbb{P}^1(\C)$ are the singular moduli, while the special points of $\mathbb{G}_m(\C)$ are defined to be the roots of unity.
\end{abstract}

\maketitle

\section{Introduction} \label{sec:Intro}

The Andr\'e--Oort conjecture says that if $S$ is a Shimura variety and $V$ is a set of special points of $S$, then the irreducible components of the Zariski closure of $V$ are special subvarieties of $S$.
There are many results in the direction of this conjecture due to Moonen, Andr\'e, Yafaev, Edixhoven, Clozel, Ullmo, and Pila among others.
A lot of these are conditional on the Generalized Riemann Hypothesis (GRH).
Assuming the GRH, we can even get effective results.
Pila has unconditional results using o-minimal geometry, however these are not effective.

The first nontrivial, unconditional, effective result was obtained by K\"uhne in \cite{Kuehne:An_effective_result_of_Andre-Oort_type} and independently by Bilu, Masser and Zannier in \cite{Bilu-Masser-Zannier}.
They study the Shimura variety $\mathbb{P}^1(\C) \times \mathbb{P}^1(\C)$, where $\mathbb{P}^1(\C)$ is the modular curve $\operatorname{SL}_2(\Z) \backslash \mathcal{H}^*$.
Here the special points are of the form $(j(\tau_1), j(\tau_2))$, where $\tau_1$ and $\tau_2$ are imaginary quadratic numbers.
A special curve is either a horizontal or vertical line, or it is defined by a modular polynomial.

Pietro Corvaja asked what happens if we look at $\mathbb{P}^1(\C) \times \mathbb{G}_m(\C)$.
Here the special points are of the form $(j(\tau), \lambda)$, where $\tau$ is an imaginary quadratic number and $\lambda$ is a root of unity.
The special curves are the horizontal and vertical lines containing at least one special point.

Our main result is the following.
If $\mathcal{C}$ is a closed algebraic curve inside $\mathbb{P}^1(\C) \times \mathbb{G}_m(\C)$ not containing any horizontal or vertical line, and $\mathcal{C}$ is defined over a number field $K$, then there are only finitely many special points on $\mathcal{C}$, and in fact we can explicitly bound the complexity of such special points.
More precisely, if $\mathcal{C}$ is the zero set of the polynomial $F(X,Y) \in K[X,Y]$, and if $(j(\tau), \lambda)$ is a special point of $\mathcal{C}$, then we can bound the discriminant $\Delta(\tau)$ and the order of $\lambda$, using only the height of $F$, the degree $[K:\Q]$, and the degrees $\deg_X F$ and $\deg_Y F$.

The structure of the paper is as follows.
In section \ref{sec:Results} we state the results.
In section \ref{sec:Prelim} we discuss some preliminary facts.
In section \ref{sec:Proofs} we prove the results.
Finally, in the last section we prove the optimality of one of our bounds in the case $K=\Q$.

\section{Results} \label{sec:Results}

Let $\mathcal{H}$ denote the complex upper half-plane.
We call $(\alpha, \lambda) \in \mathbb{P}^1(\C) \times \mathbb{G}_m(\C)$ a special point, if $\alpha = j(\tau)$ for some imaginary quadratic $\tau \in \mathcal{H}$ and $\lambda \in \C$ is a root of unity.
First we state the noneffective version of our result.
\begin{theorem} \label{thm:main-noneffective}
Let $\mathcal{C} \subseteq \mathbb{P}^1(\C) \times \mathbb{G}_m(\C)$ be a closed algebraic curve, defined over a number field.
Then $\mathcal{C}$ contains infinitely many special points if and only if either
\begin{itemize}
\item
$\mathcal{C}$ contains a horizontal line $\mathbb{P}^1(\C) \times \lambda$, where $\lambda$ is a root of unity, or
\item
$\mathcal{C}$ contains a vertical line $\alpha \times \mathbb{G}_m(\C)$, where $\alpha = j(\tau)$ for some imaginary quadratic $\tau$.
\end{itemize}
\end{theorem}
We formulate an effective version of this theorem.
Let $K$ be a number field of degree $d$ over $\Q$ with a fixed embedding into $\C$.
Let $F \in K[X,Y]$ be a nonconstant polynomial, and let $\delta_1 = \deg_X F$ and $\delta_2 = \deg_Y F$.
The equation $F(X,Y) = 0$ defines an algebraic curve in $\mathbb{A}^1(\C) \times \mathbb{A}^1(\C)$.
We assume that this curve contains no vertical or horizontal line, i.e.\ there is no $a \in \C$ such that $F(a,Y) = 0$ in $\C[Y]$, and there is no $b \in \C$ such that $F(X,b) = 0$ in $\C[X]$.
In other words, $F(X,Y)$ does not have a nonconstant divisor $f \in K[X]$ or $g \in K[Y]$.
Then clearly $\delta_1, \delta_2 > 0$.
(The condition on $F$ is satisfied e.g.\ if $\delta_1, \delta_2 > 0$ and $F$ is geometrically irreducible.)
We can restrict the above curve to $\mathbb{A}^1(\C) \times \mathbb{G}_m(\C)$, and then take the Zariski closure in $\mathbb{P}^1(\C) \times \mathbb{G}_m(\C)$.
We call the obtained curve $\mathcal{C}$.

Let $h(F)$ denote the height of the polynomial $F$ (so $h(F)$ is the absolute logarithmic Weil height of the point defined by the nonzero coefficients of $F$ in projective space, see the definition in section \ref{sec:Prelim}).
Let $(\alpha, \lambda)$ be a special point of $\mathcal{C}$, where $\alpha = j(\tau)$ for some $\tau \in \mathcal{H}$.
Let $\Delta$ denote the discriminant of the endomorphism ring of the complex elliptic curve $\C/(\Z + \Z \tau)$, and let $N$ be the smallest positive integer such that $\lambda^N = 1$.

If $n = 2^k (2m+1)$, where $k,m \in \Z_{\ge 0}$, then let
\[
c_1(n) = \begin{cases}
0 & \textrm{if } k=0 \textrm{ or } 2, \\
1 & \textrm{if } k=1, \\
-1 & \textrm{if } k \ge 3,
\end{cases}
\]
and
\[
c_2(n) = \begin{cases}
1 & \textrm{if } 4 \mid n \textrm{ or } p \mid n \textrm{ for some prime } p \equiv -1 \pmod{4} \\
0 & \textrm{otherwise.}
\end{cases}
\]
Note that $c_1(n)+c_2(n) \in \{0,1,2\}$.

The number of prime divisors of $n$ is denoted by $\omega(n)$.
Euler's totient function is denoted by $\varphi$.

\begin{theorem} \label{thm:main-effective}
In the above situation
\begin{equation} \label{eq:N-precise-bound}
\frac{\varphi(N)}{2^{\omega(N)}} \le \frac{\varphi(N)}{2^{\omega(N)-c_1(N)-c_2(N)}} \le d\delta_2,
\end{equation}
moreover
\begin{equation} \label{eq:N-upper-bound}
N < a^{1+\frac{2}{\log \log a}}
\end{equation}
and
\begin{equation} \label{eq:Delta-upper-bound}
|\Delta| < a^{2+\frac{5}{\log \log a}} (d h(F) + (d-1)(\delta_1 + \delta_2) \log 2 + 1)^2
\end{equation}
with $a = \max(8, d \delta_2)$.
\end{theorem}

Theorem \ref{thm:main-effective} clearly implies Theorem \ref{thm:main-noneffective}, because there are only finitely many special points $(\alpha, \lambda)$ with $\Delta$ and $N$ bounded.
We will prove Theorem \ref{thm:main-effective} in three steps.
In the first step we reduce the statement to the case $K = \Q$.
The second step shows that we may assume that $\Z + \Z \tau$ is an order.
Finally, in the third step we prove the theorem for $K = \Q$ and $\Z + \Z \tau$ an order.

In section \ref{sec:Optimality} we will show that the bound in \eqref{eq:N-precise-bound} is optimal if $K=\Q$.

\section{Preliminaries} \label{sec:Prelim}

The (absolute logarithmic Weil) height of a point $P = (a_0: \dotsc: a_n) \in \mathbb{P}^n_{\overline{\Q}}$ is defined by
\[
h(P) = \sum_{v \in M_K} \frac{[K_v:\Q_v]}{[K:\Q]} \log(\max_i |a_i|_v),
\]
where $K$ is any number field containing all $a_i$, $M_K$ is the set of places of $K$, and for any place $v$, $|\cdot|_v$ is the absolute value on $K$ extending a standard absolute value of $\Q$.
Similarly, the (absolute logarithmic Weil) height of a polynomial $F \in \overline{\Q}[X_1, \dotsc, X_n]$ with nonzero coefficients $c_i$ is defined by
\[
h(F) = \sum_{v \in M_K} \frac{[K_v:\Q_v]}{[K:\Q]} \log(\max_i |c_i|_v),
\]
where $K$ is a number field containing the coefficients of $F$.
We use the notation $H(F) = e^{h(F)}$.
If $F \in \Z[X_1, \dotsc, X_n]$, and the gcd of the coefficients of $F$ is $1$, then $H(F)$ is equal to the maximum of the euclidean absolute values of the coefficients of $F$.

If $K$ is a number field and $\alpha \in K$, then the (absolute logarithmic Weil) height of $\alpha$ is
\[
h(\alpha) = h(\alpha:1) = \sum_{v \in M_K} \frac{[K_v:\Q_v]}{[K:\Q]} \log \max(1, |\alpha|_v).
\]
We use the notation $H(\alpha) = e^{h(\alpha)}$.

If $\mathcal{O}$ is an order in an imaginary quadratic number field $L$, then the class number $h(\mathcal{O})$ denotes the number of equivalence classes of proper fractional ideals of $\mathcal{O}$ (see e.g.\ \cite{Cox:Primes_of_the_form_x2+ny2} or \cite{Neukirch:Algebraic_Number_Theory}).
Since $\mathcal{O}$ is an order in $L$, we can write it in the form $\Z + \Z \tau_0$ for some $\tau_0$ in $L \cap \mathcal{H}$.
Then the discriminant of the order $\mathcal{O}$ is $D(\mathcal{O}) = -4 (\Impart \tau_0)^2$ (see e.g.\ \S 7, Ch.\ 2 of \cite{Cox:Primes_of_the_form_x2+ny2}).
This is a negative integer congruent to $0$ or $1$ modulo $4$.

We introduce the notation $j(\Lambda)$, where $\Lambda \subseteq \C$ is a lattice.
There are complex numbers $\omega_1, \omega_2 \in \C^{\times}$ such that $\Lambda = \Z \omega_1 + \Z \omega_2$ and $\tau = \frac{\omega_2}{\omega_1} \in \mathcal{H}$.
Then we define $j(\Lambda)$ to be $j(\tau)$.
The modularity of the $j$-function ensures that $j(\Lambda)$ is well defined.

We will use the following estimate from \cite{Wu:A_note_on_Andre-Oort} for the $j$-function.
Similar estimates can be found in \cite{Bilu-Parent}.
\begin{proposition} \label{prop:j-estimate}
If $\tau \in \mathcal{H}$ and $\Impart \tau > \frac{1}{2\pi} \log 6912$, then $\frac{1}{2} \le \left|\frac{j(\tau)}{e^{-2\pi i \tau}}\right| \le 2$.
\end{proposition}

Let $G$ be a group.
We say that $G$ is an elementary abelian $2$-group, if $G$ is a finite abelian group such that every element of $G$ has order at most $2$, or in other words, if $G \cong (\Z/2\Z)^s$ for some integer $s \ge 0$.
We say that $G$ is a generalized dihedral group, if $G$ is isomorphic to a semidirect product $H \rtimes (\Z/2\Z)$, where $H$ is a group, and the nontrivial element of $\Z/2\Z$ acts on $H$ by inverting elements (so $H$ must be abelian).
The following lemma describes finite abelian quotients of generalized dihedral groups.
\begin{lemma} \label{lemma:gen-dihedral-group}
A finite abelian quotient of a generalized dihedral group is an elementary abelian $2$-group.
\end{lemma}
\begin{proof}
Let $G = H \rtimes \{\pm 1\}$ be a generalized dihedral group, where $H$ is an abelian group.
We use multiplicative notation for the groups $G$, $H$ and $\{\pm 1\}$.
Let $N$ be a normal subgroup of $G$ such that $G/N$ is a finite abelian group.
Then $G' \subseteq N$.
To prove that every element of $G/N$ has order at most two, it is enough to show that $G^2 \subseteq G'$.
The multiplication in $G$ is defined by $(h_1,a_1)(h_2,a_2) = (h_1 h_2^{a_1}, a_1 a_2)$.
So
\[
(h,1)^2 = (h^2,1) = (h,1)(1,-1)(h,1)^{-1}(1,-1)^{-1} \in G'
\]
and $(h,-1)^2 = 1 \in G'$ for every $h \in H$.
\end{proof}

The Galois group of the extension $\Q(\lambda)/\Q$ is isomorphic to $(\Z/N\Z)^{\times}$.
We need to know the size of the group $((\Z/N\Z)^{\times})^2$.
The following lemma solves this problem.
\begin{lemma} \label{lemma:(Z/NZ)*/(Z/NZ)*^2}
If $N$ is a positive integer, then
\[
|((\Z/N\Z)^{\times})/((\Z/N\Z)^{\times})^2| = 2^{\omega(N) - c_1(N)}.
\]
\end{lemma}
\begin{proof}
Let $N = \prod_{i=1}^s p_i^{a_i}$, where $p_1, \dotsc, p_s$ are distinct primes and $a_i \in \Z_{\ge 1}$.
Using the Chinese remainder theorem we obtain the isomorphism
\[
((\Z/N\Z)^{\times})/((\Z/N\Z)^{\times})^2 \cong \prod_{i=1}^s ((\Z/p_i^{a_i}\Z)^{\times})/((\Z/p_i^{a_i}\Z)^{\times})^2.
\]
Let $p$ be a prime and $a$ a positive integer.
It is well known that
\[
(\Z/p^a\Z)^{\times} \cong
\begin{cases}
\Z/(p^{a-1}(p-1)\Z) & \textrm{if } p \textrm{ is odd,} \\
1 & \textrm{if } p=2 \textrm{ and } a=1, \\
\Z/2\Z & \textrm{if } p=2 \textrm{ and } a=2, \\
(\Z/2\Z) \times (\Z/2^{a-2}\Z) & \textrm{if } p=2 \textrm{ and } a \ge 3.
\end{cases}
\]
Hence
\[
(\Z/p^a\Z)^{\times} / ((\Z/p^a\Z)^{\times})^2 \cong
\begin{cases}
\Z/2\Z & \textrm{if } p \textrm{ is odd,} \\
1 & \textrm{if } p=2 \textrm{ and } a=1, \\
\Z/2\Z & \textrm{if } p=2 \textrm{ and } a=2, \\
(\Z/2\Z) \times (\Z/2\Z) & \textrm{if } p=2 \textrm{ and } a \ge 3.
\end{cases}
\]
The statement of the lemma follows by applying this result for each $p_i$ and $a_i$.
\end{proof}

The following proposition shows that the intersection of a field $\Q(j(\mathcal{O}))$ and a cyclotomic field cannot be too big.
\begin{proposition} \label{prop:bound-for-j(O)-cap-Q(lambda)}
Let $L$ be an imaginary quadratic number field, $\mathcal{O} \subseteq L$ an order, $N$ a positive integer and $\lambda \in \C$ a primitive $N^{\textrm{th}}$ root of unity.
Then
\begin{equation} \label{eq:bound-for-j(O)-cap-Q(lambda)}
[\Q(j(\mathcal{O})) \cap \Q(\lambda) : \Q] \le 2^{\omega(N) - c_1(N) - c_2(N)}.
\end{equation}
\end{proposition}
\begin{proof}
Let $N = 2^k p_1^{\alpha_1} \dotsm p_r^{\alpha_r}$, where $p_1, \dotsc, p_r$ are distinct odd primes, $k \in \Z_{\ge 0}$ and $\alpha_1, \dotsc, \alpha_r \in \Z_{\ge 1}$.
Let
\begin{equation} \label{eq:def-of-field-E}
E = \begin{cases}
\Q(\sqrt{p_1^*}, \dotsc, \sqrt{p_r^*}) & \textrm{if } k=0 \textrm{ or } 1, \\
\Q(\sqrt{-1}, \sqrt{p_1^*}, \dotsc, \sqrt{p_r^*}) & \textrm{if } k=2, \\
\Q(\sqrt{-1}, \sqrt{2}, \sqrt{p_1^*}, \dotsc, \sqrt{p_r^*}) & \textrm{if } k \ge 3,
\end{cases}
\end{equation}
where $p^* = (-1)^{\frac{p-1}{2}} p$ for every odd prime $p$.
We will show that $E \subseteq \Q(\lambda)$.
Let $\zeta_n = e^{\frac{2\pi i}{n}}$ for every positive integer $n$, then $\Q(\lambda) = \Q(\zeta_N)$.
It is a basic fact from the theory of Gauss sums that
\[
\sqrt{p^*} = \sum_{a=0}^{p-1} \zeta_p^{a^2} \in \Q(\zeta_p)
\]
holds for every odd prime $p$.
Moreover $\sqrt{-1} \in \Q(\zeta_4)$ and $\sqrt{2} \in \Q(\zeta_8)$, therefore $E \subseteq \Q(\lambda)$.

It is well known that $\Q(\lambda)/\Q$ is an abelian extension with Galois group
\[
G = \Gal(\Q(\lambda)/\Q) \cong (\Z/N\Z)^{\times}.
\]
We will prove that $E = \Q(\lambda)^{G^2}$.
If $x \in \Q$, $\sqrt{x} \in E$ and $\sigma \in G$, then $\sigma(\sqrt{x}) = \pm \sqrt{x}$, hence $\sigma^2(\sqrt{x}) = \sqrt{x}$, therefore $\sqrt{x} \in \Q(\lambda)^{G^2}$.
Applying this to the generators of $E$, we obtain $E \subseteq \Q(\lambda)^{G^2}$.
Moreover by Lemma \ref{lemma:(Z/NZ)*/(Z/NZ)*^2}
\[
[E:\Q] = 2^{\omega(N) - c_1(N)} = \frac{|G|}{|G|^2} = \frac{[\Q(\lambda):\Q]}{[\Q(\lambda):\Q(\lambda)^{G^2}]} = [\Q(\lambda)^{G^2}:\Q]
\]
holds, proving $E = \Q(\lambda)^{G^2}$.

Let $E' = L(j(\mathcal{O})) \cap \Q(\lambda)$ and $E'' = \Q(j(\mathcal{O})) \cap \Q(\lambda)$, then $E'' \subseteq E'$.
We will prove that $E' \subseteq E$.
Lemma 9.3 and Theorem 11.1 in \cite{Cox:Primes_of_the_form_x2+ny2} show that the extension $L(j(\mathcal{O}))/\Q$ is Galois, and $\Gal(L(j(\mathcal{O}))/\Q)$ is a generalized dihedral group.
So $E' = L(j(\mathcal{O})) \cap \Q(\lambda)$ is a finite Galois extension of $\Q$.
Consider the following isomorphisms of groups:
\[
\Gal(L(j(\mathcal{O}))/\Q)/\Gal(L(j(\mathcal{O}))/E') \cong \Gal(E'/\Q) \cong G/\Gal(\Q(\lambda)/E').
\]
The second isomorphism shows that $\Gal(E'/\Q)$ is a finite abelian group, while the first shows that it is a quotient of a generalized dihedral group.
So $\Gal(E'/\Q) \cong G/\Gal(\Q(\lambda)/E')$ is an elementary abelian $2$-group by Lemma \ref{lemma:gen-dihedral-group}.
Then $G^2 \subseteq \Gal(\Q(\lambda)/E')$, so
\[
E'' \subseteq E' \subseteq \Q(\lambda)^{G^2} = E.
\]

Since $[E'':\Q] \mid [E:\Q] = 2^{\omega(N)-c_1(N)}$, it is enough to prove that $E'' \neq E$ if $c_2(N)=1$.
So suppose $c_2(N) = 1$.
Then $\sqrt{-1} \in E$ or $\sqrt{-p} \in E$ for some prime $p$, hence $E \not\subseteq \R$.
However $j(\mathcal{O}) \in \R$, because $\overline{j(\mathcal{O})} = j(\overline{\mathcal{O}}) = j(\mathcal{O})$.
So $E'' \subseteq \R$, therefore $E'' \neq E$.
\end{proof}

\begin{remark} \label{remark:j(O)-cap-Q(lambda)-bound-optimal}
The bound in Proposition \ref{prop:bound-for-j(O)-cap-Q(lambda)} is optimal.
To see this, take an imaginary quadratic number field $L$, a positive integer $N$ and a primitive $N^{\textrm{th}}$ root of unity $\lambda$.
We will show that there is an order $\mathcal{O}$ in $L$ such that equality holds in \eqref{eq:bound-for-j(O)-cap-Q(lambda)}.
Let $N = 2^k p_1^{\alpha_1} \dotsm p_r^{\alpha_r}$, where $p_1, \dotsc, p_r$ are distinct odd primes, $k \in \Z_{\ge 0}$ and $\alpha_1, \dotsc, \alpha_r \in \Z_{\ge 1}$.
Define the field $E$ as in \eqref{eq:def-of-field-E}.
We have seen during the proof of Proposition \ref{prop:bound-for-j(O)-cap-Q(lambda)} that $[E:\Q] = 2^{\omega(N)-c_1(N)}$ and $L(j(\mathcal{O})) \cap \Q(\lambda) \subseteq E$ for every order $\mathcal{O}$ in $L$.

The field $LE$ is generated over $\Q$ by the square roots of a few integers, thus it is a Galois extension with $\Gal(LE/\Q) \cong (\Z/2\Z)^s$ for some $s \ge 1$.
So $\Gal(LE/\Q)$ is a generalized dihedral group and $LE/L$ is abelian, hence by Corollary 11.35 in \cite{Cox:Primes_of_the_form_x2+ny2} there is an order $\mathcal{O}$ in $L$ such that $LE \subseteq L(j(\mathcal{O}))$.
So $E \subseteq L(j(\mathcal{O})) \cap \Q(\lambda) \subseteq E$, hence $E = L(j(\mathcal{O})) \cap \Q(\lambda)$.
Note that $\Q(j(\mathcal{O})) = \R \cap L(j(\mathcal{O}))$.
If $c_2(N)=0$, then
\[
E = \R \cap E = \R \cap L(j(\mathcal{O})) \cap \Q(\lambda) = \Q(j(\mathcal{O})) \cap \Q(\lambda).
\]
If $c_2(N)=1$, then $E \not\subseteq \R$, so $\Q(j(\mathcal{O})) \cap E = \Q(j(\mathcal{O})) \cap \Q(\lambda)$ and $\Q(j(\mathcal{O})) E = L(j(\mathcal{O}))$, therefore
\[
[E:\Q(j(\mathcal{O})) \cap \Q(\lambda)] = [L(j(\mathcal{O})):\Q(j(\mathcal{O}))] = 2
\]
by Theorem 1.12, Ch.\ VI, \S 1 in \cite{Lang:Algebra}.
Either way we have
\[
[\Q(j(\mathcal{O})) \cap \Q(\lambda):\Q] = 2^{\omega(N) - c_1(N) - c_2(N)}.
\]
\end{remark}

The following proposition gives us an upper bound for $n$ provided we have an upper bound for $\frac{\varphi(n)}{2^{\omega(n)}}$.
\begin{proposition} \label{prop:phi(n)/2^omega(n)<a}
Let $n$ be a positive integer.
If $\frac{\varphi(n)}{2^{\omega(n)}} \le a$ for a real number $a \ge 8$, then
\[
n < a^{1+\frac{2}{\log \log a}}.
\]
\end{proposition}
\begin{proof}
We need to prove that $n < u(a)$, where
\[
u \colon (e, \infty) \to \R, \quad x \mapsto x^{1 + \frac{2}{\log \log x}}.
\]
Using derivation one can check that the function $u(x)$ is minimal at $x = e^{e^{\sqrt{3}-1}} \approx 7.99919$, and it increases in the interval $[8, \infty)$.
So $u(a) \ge u(8) > 2345$, hence we may assume that $n \ge 2346$.

Let $P_i$ denote the $i^\textrm{th}$ prime number, and let $M_s = \prod_{i=1}^s P_i$ and $A_s = \prod_{i=1}^s \frac{P_i}{P_i-1}$ for every positive integer $s$.
Lemma 14 in \cite{Rosser-Schoenfeld:Approx_formulas} says that $\frac{m}{\varphi(m)} \le A_s$ holds for all positive integers $s$ and $m$ such that $m < M_{s+1}$.
Clearly $m < M_{s+1}$ also implies $\omega(m) \le s$.
Suppose $2346 \le n < M_6 = 30030$, then $\frac{n}{\varphi(n)} \le 2 \cdot \frac{3}{2} \cdot \frac{5}{4} \cdot \frac{7}{6} \cdot \frac{11}{10} = \frac{77}{16}$ and $\omega(n) \le 5$, therefore $a \ge \frac{\varphi(n)}{2^{\omega(n)}} \ge \frac{n}{154} > 8$.
So it is enough to prove that $n < u(\frac{n}{154})$.
This follows from the inequality $\frac{u(x)}{x} = e^{2 \frac{\log x}{\log \log x}} \ge e^{2e} > 154$, which is true for every $x > e$, because $\frac{y}{\log y} \ge e$ for every $y>1$.
So we may assume that $n \ge M_6$.

There is a unique integer $s \ge 6$ such that $M_s \le n < M_{s+1}$.
Then $\omega(n) \le s$ and $\frac{n}{\varphi(n)} \le A_s$, so $a \ge \frac{\varphi(n)}{2^{\omega(n)}} \ge \frac{n}{2^s A_s}$.
Note that
\[
\frac{n}{2^s A_s} \ge \frac{M_s}{2^s A_s} = \prod_{i=1}^s \frac{P_i-1}{2} \ge \frac{1}{2} \cdot \frac{2}{2} \cdot \frac{4}{2} \cdot \frac{6}{2} \cdot \frac{10}{2} \cdot \frac{12}{2} = 90 > e^e > 8.
\]
Thus $u(a) \ge u(\frac{n}{2^s A_s})$, so it is enough to prove $n < u(\frac{n}{2^s A_s})$.
Let $y = \frac{n}{2^s A_s}$, then we need to show that $\frac{u(y)}{y} > 2^s A_s$, or equivalently, that $\frac{z}{\log z} > \frac{1}{2} \log(2^s A_s)$, where $z = \log y$.
The function $\frac{x}{\log x}$ is increasing in the interval $[e, \infty)$, and $z \ge \log(\frac{M_s}{2^s A_s}) > e$, so it is enough to show that
\begin{equation} \label{eq:2^s-As-Ms}
\frac{1}{2} \log(2^s A_s) < \frac{c_s}{\log c_s},
\end{equation}
where $c_s = \log \frac{M_s}{2^s A_s}$.
One can easily check this in each of the cases $s = 6, 7, \dotsc, 12$, so we may assume that $s \ge 13$.
Then $\log M_s > s \log s$ by Theorem 4 in \cite{Robin}.
We have $A_m < (\frac{e}{2})^m$ for every $m \ge 6$, because $A_6 < (\frac{e}{2})^6$ and $\frac{A_{m}}{A_{m-1}} = \frac{P_m}{P_m-1} \le \frac{17}{16} < \frac{e}{2}$ for every $m \ge 7$.
Thus $2^s A_s < e^s$ and hence
\[
c_s > \log M_s - s > s (\log s - 1) > e.
\]
The function $\frac{x}{\log x}$ is increasing in the interval $[e, \infty)$, so we get
\[
\frac{c_s}{\log c_s} \ge \frac{s(\log s - 1)}{\log s + \log(\log s - 1)}.
\]
Using the inequality $\log x \le x-1$ for $x = \log s - 1$, we obtain
\[
\frac{s(\log s - 1)}{\log s + \log(\log s - 1)} \ge \frac{s}{2} > \frac{1}{2} \log(2^s A_s),
\]
which proves \eqref{eq:2^s-As-Ms}.
\end{proof}

\section{Proof of the results} \label{sec:Proofs}

We have already observed that Theorem \ref{thm:main-effective} implies Theorem \ref{thm:main-noneffective}.
We start the proof of Theorem \ref{thm:main-effective} by reducing the statement to the case $K=\Q$.
Let $S$ be the set of embeddings of $K$ into $\overline{K} = \overline{\Q}$.
Then $|S|=d$, and the polynomial $F'(X,Y) = \prod_{\sigma \in S} F^{\sigma}(X,Y)$ is in $\Q[X,Y]$.
One way to see this is to take any $\theta \in \Gal(L/\Q)$, where $L/\Q$ is a finite Galois extension such that $K \subseteq L$, and then observe that $F'^{\theta} = \prod_{\sigma \in S} F^{\theta \sigma} = \prod_{\eta \in S} F^{\eta} = F'$, because multiplying from the left by $\theta$ only permutes the elements of $S$.
The conditions for $F$ are also satisfied by $F'$: it is nonconstant, and its zero set contains no horizontal or vertical line.
Moreover $\delta_1' = \deg_X F' = d \delta_1$, $\delta_2' = \deg_Y F' = d \delta_2$, and $F'(\alpha,\lambda) = 0$.

Theorem 1.6.13 and Remark 1.6.14 in \cite{Bombieri-Gubler} give bounds for the height of a product of polynomials.
Using these results and the fact that $h(F^{\sigma}) = h(F)$ for every $\sigma \in S$, we get the bound
\[
h(F') \le d h(F) + (d-1) (\delta_1+\delta_2) \log 2.
\]
Applying the case of $K = \Q$ for $F'$ and $(\alpha, \lambda)$, we get the bounds in the statement of the theorem.
From now on we assume that $K=\Q$.

Now we make the second step.
Let us define the lattice $\Lambda = \Z + \Z \tau$ and the order $\mathcal{O} = \operatorname{End}(\C/\Lambda)$.
Since $\mathcal{O}$ is an order in $\Q(\tau)$, we can write it as $\Z + \Z \tau_0$ for some $\tau_0$ in $\Q(\tau) \cap \mathcal{H}$.
Then $\Lambda$ defines a proper fractional ideal $\mathfrak{a}$ of the order $\mathcal{O}$ (here proper means that the endomorphism ring of $\mathfrak{a}$ is $\mathcal{O}$).
Let $h(\mathcal{O})$ be the class number of $\mathcal{O}$, and let $\mathfrak{a}_1, \dotsc, \mathfrak{a}_{h(\mathcal{O})}$ be the proper fractional ideals of $\mathcal{O}$ representing the different classes.
Then $j(\mathcal{O})$ is an algebraic integer of degree $h(\mathcal{O})$, and $j(\mathfrak{a}_1)$, \ldots, $j(\mathfrak{a}_{h(\mathcal{O})})$ are the conjugates of it over $\Q$ (see Theorem 11.1 and Proposition 13.2 in \cite{Cox:Primes_of_the_form_x2+ny2}, or Theorem 5, \S 3, Ch.\ 10 in \cite{Lang:Elliptic_functions}).
So in particular $\alpha = j(\tau) = j(\mathfrak{a})$ and $\alpha_0 = j(\tau_0) = j(\mathcal{O})$ are conjugates over $\Q$, hence there is an automorphism $\sigma \in \operatorname{Aut}(\overline{\Q})$ such that $\sigma(\alpha) = \alpha_0$.
Applying $\sigma$ to the equation $F(\alpha,\lambda) = 0$, we get that $F(\alpha_0, \sigma(\lambda)) = 0$.
This shows that $(j(\tau_0), \sigma(\lambda))$ is a special point on $\mathcal{C}$, where the discriminant of $\operatorname{End}(\C/(\Z+\Z\tau_0))$ is $\Delta$, and $\sigma(\lambda)$ is a primitive $N^{\textrm{th}}$ root of unity.
The second step is finished, because we can replace the special point $(\alpha, \lambda)$ with $(j(\tau_0), \sigma(\lambda))$.
From now on we assume that $\Z + \Z \tau$ is an order.

Now we make the third step.
Since $K = \Q$, we have $d=1$ and $a = \max(8, \delta_2)$.
The polynomial $g(Y) = F(\alpha,Y) \in \Q(\alpha)[Y]$ is nonzero, because the zero set of $F$ contains no vertical line.
Moreover $g(\lambda) = 0$, so
\[
[\Q(\alpha,\lambda):\Q(\alpha)] \le \deg g \le \delta_2.
\]
Theorem 1.12, Ch.\ VI, \S 1 in \cite{Lang:Algebra} implies that
\[
[\Q(\lambda):\Q(\alpha) \cap \Q(\lambda)] = [\Q(\alpha, \lambda):\Q(\alpha)] \le \delta_2.
\]
Applying Proposition \ref{prop:bound-for-j(O)-cap-Q(lambda)}, we obtain \eqref{eq:N-precise-bound}.
Now Proposition \ref{prop:phi(n)/2^omega(n)<a} implies \eqref{eq:N-upper-bound}.

To prove \eqref{eq:Delta-upper-bound} we need to show that
\begin{equation} \label{eq:Delta-bound-K=Q}
\sqrt{|\Delta|} < a^{1+\frac{2.5}{\log \log a}} (h(F)+1).
\end{equation}
Here $\Delta = -4 (\Impart \tau)^2$ is the discriminant of the order $\mathcal{O} = \Z + \Z \tau$.
If $|\Delta| \le 10^6$, then
\[
\sqrt{|\Delta|} \le 1000 < a^{1+\frac{5/2}{\log\log a}} (h(F)+1),
\]
so we may assume that $|\Delta| > 10^6$.
Then $\Impart \tau = \frac{\sqrt{|\Delta|}}{2} > 500$, and from Proposition \ref{prop:j-estimate} we deduce that $\frac{|\alpha|}{e^{2\pi \Impart \tau}} \in [\frac{1}{2},2]$.
This leads to
\begin{equation} \label{eq:Delta-log-alpha-bound}
\log(|\alpha|-1) > \log|\alpha| - 1 \ge 2\pi \Impart \tau - \log 2 - 1 > 6 \Impart \tau = 3 \sqrt{|\Delta|},
\end{equation}
because $|\alpha| \ge \frac{1}{2} e^{2\pi \Impart \tau} > e^{500}$ and $\Impart \tau > 500 > \frac{\log 2 + 1}{2\pi - 6}$.

We multiply $F$ by a nonzero rational number, so that $F$ will have integer coefficients with gcd equal to $1$.
Then the maximum of the euclidean absolute values of the coefficients of $F$ is $H(F) = e^{h(F)}$.
Let $F = \sum_{i=0}^{\delta_1} g_i(Y) X^i$, where $g_i(Y) \in \Z[Y]$.
Here each $g_i$ has degree at most $\delta_2$.
Since $F(X,\lambda) \in \C[X]$ is a nonzero polynomial, $g_i(\lambda) \neq 0$ for some $i$.
Let $m$ be the maximal such $i$.
Then $\sum_{i=0}^m g_i(\lambda) \alpha^i = F(\alpha, \lambda) = 0$, and here $g_m(\lambda) \neq 0$.
So $-g_m(\lambda) \alpha^m = \sum_{i=0}^{m-1} g_i(\lambda) \alpha^i$.
By the triangle inequality we find that
\[
|g_m(\lambda)| \cdot |\alpha|^m \le \sum_{i=0}^{m-1} |g_i(\lambda)| \cdot |\alpha|^i.
\]
Note that $|g_i(\lambda)| \le (\delta_2+1) H(F)$ for every $i$.
So
\[
|g_m(\lambda)| \le (\delta_2+1) H(F) \sum_{i=0}^{m-1} |\alpha|^{i-m} < (\delta_2+1) H(F) \sum_{j=1}^{\infty} |\alpha|^{-j} = \frac{(\delta_2+1) H(F)}{|\alpha|-1},
\]
therefore
\begin{equation} \label{eq:upper-bound-for-log-alpha}
\log(|\alpha|-1) < h(F) + \log(\delta_2 + 1) - \log |g_m(\lambda)|.
\end{equation}
We need a lower bound for $|g_m(\lambda)|$.
If $g_m$ is nonconstant, then we apply Theorem A.1 in \cite{Bugeaud} with $P(X) = g_m(X)$, $Q(X) = X^N-1$ and $\beta = \lambda$.
We obtain
\begin{equation} \label{eq:upper-bound-for-gm(lambda)}
|g_m(\lambda)| \ge (\delta_2 + 1)^{1-N} (N+1)^{-\delta_2/2} H(F)^{1-N},
\end{equation}
because $\deg g_m \le \delta_2$ and all the coefficients of $g_m$ have euclidean absolute value at most $H(F)$.
If $g_m$ is constant, then $|g_m(\lambda)| \ge 1$, hence \eqref{eq:upper-bound-for-gm(lambda)} is still valid.
Therefore
\[
-\log |g_m(\lambda)| \le (N-1) (h(F) + \log(\delta_2 + 1)) + \frac{\delta_2}{2} \log(N+1),
\]
and together with \eqref{eq:Delta-log-alpha-bound} and \eqref{eq:upper-bound-for-log-alpha}, this implies that
\[
\sqrt{|\Delta|} < \frac{N}{3} (h(F) + \log(\delta_2 + 1)) + \frac{\delta_2}{6} \log(N+1).
\]
Here $\frac{N}{3} h(F) \le a^{1+\frac{2.5}{\log\log a}} h(F)$ and $\delta_2 \le a$, so to prove \eqref{eq:Delta-bound-K=Q} it is enough to show
\[
\frac{N}{3} \log(a+1) + \frac{a}{6} \log(N+1) < a^{1+\frac{2.5}{\log\log a}}.
\]
Since $N < a^{1+\frac{2}{\log \log a}}$ and
\[
N+1 < a^{1+\frac{2}{\log \log a}} + 1 < (a+1)^{1+\frac{2}{\log \log a}} < (a+1)^4,
\]
it is enough to prove
\[
\frac{1}{3} \left(a^{\frac{2}{\log\log a}} + 2\right) \log(a+1) < a^{\frac{2.5}{\log\log a}}.
\]
Here $2 < \frac{1}{5} a^{\frac{2}{\log\log a}}$ and $\log(a+1) < \frac{5}{4} \log a$, so
\[
\frac{1}{3} \left(a^{\frac{2}{\log\log a}} + 2\right) \log(a+1) < \frac{1}{2} a^{\frac{2}{\log\log a}} \log a.
\]
Hence it is enough to prove
\[
\frac{1}{2} \log a < a^{\frac{1}{2\log\log a}}.
\]
Taking logarithms and writing $x = \frac{1}{2} \log a > 1$, we get that this is equivalent to
\[
(\log x) \log(2x) < x.
\]
Note that $\frac{1}{2}\log x = \log(\sqrt{x}) \le \frac{1}{e} \sqrt{x}$, so $\log x \le \frac{2}{e} \sqrt{x}$, and similarly $\log(2x) \le \frac{2}{e} \sqrt{2x}$.
Hence $(\log x) \log(2x) \le \frac{4\sqrt{2}}{e^2} x < x$.

\section{Optimality of the bound for $N$ if $K=\mathbb{Q}$} \label{sec:Optimality}

We will show that if $K=\Q$, then the bound in \eqref{eq:N-precise-bound} is optimal in the following sense.
Let $N$ and $\delta_2$ be positive integers such that $\frac{\varphi(N)}{2^{\omega(N) - c_1(N) - c_2(N)}} \le \delta_2$.
Let $L$ be an imaginary quadratic number field and $\lambda \in \C$ a primitive $N^{\textrm{th}}$ root of unity.
Then there is an order $\mathcal{O}$ in $L$ and a polynomial $F(X,Y) \in \Q[X,Y]$ such that the zero set of $F$ does not contain a horizontal or vertical line, $\deg_Y F = \delta_2$, and $F(j(\mathcal{O}), \lambda) = 0$.

We may assume that $\delta_2 = \frac{\varphi(N)}{2^{\omega(N) - c_1(N) - c_2(N)}}$, because the expression on the right hand side is a positive integer, and one can multiply $F$ by $(X+Y)^s$ for any $s \in \Z_{\ge 0}$ to increase the degree in $Y$.
Note that $\omega(N) - c_1(N) - c_2(N) \ge 0$, so $\delta_2 \le \varphi(N)$.
First suppose that $\delta_2 = \varphi(N)$.
Take any order $\mathcal{O}$ in $L$.
Then
\[
F(X,Y) = f_{\alpha}(X) + f_{\lambda}(Y)
\]
works, where $f_{\alpha}$ and $f_{\lambda}$ denote the minimal polynomials over $\Q$ of $\alpha = j(\mathcal{O})$ and $\lambda$.
So we may assume that $\delta_2 < \varphi(N)$.

By Remark \ref{remark:j(O)-cap-Q(lambda)-bound-optimal}, there is an order $\mathcal{O}$ in $L$ such that
\[
[\Q(\alpha) \cap \Q(\lambda) : \Q] = 2^{\omega(N)-c_1(N)-c_2(N)}
\]
with $\alpha = j(\mathcal{O})$.
Then Theorem 1.12, Ch.\ VI, \S 1 in \cite{Lang:Algebra} implies that
\[
[\Q(\alpha, \lambda) : \Q(\alpha)] = [\Q(\lambda) : \Q(\alpha) \cap \Q(\lambda)] = \frac{\varphi(N)}{2^{\omega(N) - c_1(N) - c_2(N)}} = \delta_2.
\]
So there is a polynomial $g(Y) \in \Q(\alpha)[Y]$ of degree $\delta_2$ such that $g(\lambda) = 0$.
Writing each coefficient of $g$ as a polynomial of $\alpha$, we get a polynomial $G(X,Y) \in \Q[X,Y]$ such that $g(Y) = G(\alpha, Y)$ and $\deg_Y G = \delta_2$.
Then $G(\alpha, \lambda) = 0$, hence there is an irreducible factor $F \in \Q[X,Y]$ of $G$ such that $F(\alpha, \lambda) = 0$ and $\deg_Y F \le \delta_2$.
Suppose the zero set of $F$ contains a horizontal or vertical line.
Since $F$ is irreducible, this means that $F(X,Y) \in \Q[X]$ or $F(X,Y) \in \Q[Y]$.
In the first case $F(\alpha, Y) = F(\alpha, \lambda) = 0$, so $g(Y) = G(\alpha, Y) = 0$, which is impossible.
In the second case $F(X,\lambda) = F(\alpha, \lambda) = 0$, hence
\[
\varphi(N) \le \deg_Y F \le \delta_2 < \varphi(N),
\]
which is also impossible.
So the zero set of $F$ contains no horizontal or vertical line, moreover $F(\alpha,\lambda) = 0$ and $\deg_Y F \le \delta_2$.
Theorem \ref{thm:main-effective} says that $\deg_Y F \ge \delta_2$, so in fact $\deg_Y F = \delta_2$.

\subsection*{Acknowledgements}
This paper has its origins in the author's Ph.D.\ studies under the supervision of Gisbert W\"ustholz at ETH Z\"urich.
Therefore the author thanks Gisbert W\"ustholz for introducing him to this field, and for all the helpful discussions.

\bibliographystyle{plain}
\bibliography{Andre-Oort-P1xGm}

\end{document}